\documentclass[11pt]{article}

\usepackage[margin=1in]{geometry}
\usepackage{amsmath,amssymb,amsthm,mathtools}
\usepackage{microtype}
\usepackage{enumitem}
\usepackage[colorlinks=true,linkcolor=blue,citecolor=blue,urlcolor=blue]{hyperref}

\newtheorem{theorem}{Theorem}[section]
\newtheorem{lemma}[theorem]{Lemma}

\newtheorem{remark}[theorem]{Remark}
\newtheorem{conjecture}[theorem]{Conjecture}

\newcommand{\R}{\mathbb{R}}
\newcommand{\one}{\mathbf{1}}

\newcommand{\norm}[1]{\left\lVert #1\right\rVert}
\newcommand{\pos}[1]{\left(#1\right)_{+}}

\begin{document}

\title{Characterizing the equality case in \\Brouwer's inequality for Laplacian eigenvalues}

\author{\Large{Yuhang Cui}\thanks{E-mail: yuhang\_cui@qq.com.}  
\Large{ and Xiaodan Chen}\thanks{Corresponding author. E-mail: xdchen@gxu.edu.cn.}}

\date{School of Mathematics $\&$ Center for Applied Mathematics of Guangxi, \\
Guangxi University, Nanning 530004, Guangxi, P. R. China}

\maketitle

\begin{abstract}
Brouwer conjectured that the sum of the $k$ largest Laplacian eigenvalues of an $n$-vertex graph
is less than or equal to the number of its edges plus $\binom{k+1}{2}$ for every $k\in \{1,2,\dots,n\}$,
which has been confirmed by Kothari and Tudose (2026) recently.
In this note, we characterize the equality case in this inequality.
Our main result is that for every $n$-vertex graph $G=(V,E)$ and for every $k\in \{1,2,\dots,n-1\}$, the equality
\[
  \sum_{i=1}^k\mu_i(G)=|E(G)|+\binom{k+1}{2}
\]
holds if and only if $G$ is a threshold graph with clique number $k+1$,
where $\mu_1(G)\geq \mu_2(G)\geq \cdots\geq \mu_{n}(G)$ are the Laplacian eigenvalues of $G$.
This, together with the confirmed Brouwer's conjecture, would yield a complete solution to the full Brouwer's conjecture posed by Li and Guo (2022).
Our proof relies on the projection method of Kothari and Tudose and shows directly that the equality case can occur only for threshold graphs.
\end{abstract}


\section{Introduction}

Throughout this note, all graphs are finite, undirected, and simple.
Let $G$ be a graph with $n$ vertices and $m$ edges.
The Laplacian matrix of $G$ is given by
\[
  L(G)=D(G)-A(G),
\]
where $D(G)$ is the diagonal matrix of vertex degrees of $G$ and $A(G)$ is the adjacency matrix of $G$.
It is known that $L(G)$ is positive semi-definite and hence, its eigenvalues are real and nonnegative,
which can usually be ordered as
\[
  \mu_1(G)\geq \mu_2(G)\ge\cdots\geq \mu_n(G)=0.
\]

In this note, we are concerned with the following spectral parameter:
\[
  S_k(G):=\sum_{i=1}^k\mu_i(G),\,\,\, \textrm{$1\leq k\leq n$},
\]
which turns out to be closely related with the (conjugate) degree sequence of a graph $G$.
Indeed, the Grone--Merris conjecture, proved by Bai and now known as the Grone--Merris--Bai theorem, states that
the Laplacian eigenvalue sequence of a graph is majorized by its conjugate degree sequence \cite{Bai,GM}.
Motivated by this prominent conjecture, Brouwer \cite{BH} proposed another interesting conjecture (later known as Brouwer's conjecture):

\begin{conjecture}[Brouwer's conjecture \cite{BH}]
For every graph $G$ with $n$ vertices and $m$ edges and for every $k\in \{1,2,\dots,n\}$,
\begin{equation}\label{eq-1-1}
  S_k(G)\leq m+\binom{k+1}{2}.
\end{equation}
\end{conjecture}
\noindent 

As one of the fundamental and challenging problems in the field of spectral graph theory,
Brouwer's conjecture has attracted considerable attention in the past nearly 20 years
(see \cite{Berndsen,Chen1,Chen2,Cooper,DuZhou,Ganie1,Ganie2,HMT,Helmberg,Kumar,Lew1,Lew2,Lew3,Mayank,Rocha1,Rocha2,Torres1,Torres2,WangK2,WangS} for partial solutions),
until Kothari and Tudose \cite{KT} proved this conjecture in full recently.

A subsequent concern is to characterize the equality case in \eqref{eq-1-1},
which motivates Li and Guo \cite{LG} to propose the full version of Brouwer's conjecture; 
see also the discussion of Chen and Zi \cite{CZ}.

\begin{conjecture}[The full Brouwer's conjecture \cite{LG}]
For every graph $G$ with $n$ vertices and $m$ edges and for every $k\in \{1,2,\dots,n-1\}$,
\begin{equation*}
  S_k(G)\leq m+\binom{k+1}{2},
\end{equation*}
with equality holding if and only if $G$ is a threshold graph \footnote{A graph is said to be threshold if and only if 
it can be constructed through an iterative process which starts with an isolated vertex,
and where, at each step, either a new isolated vertex is added, or a new dominating vertex
(i.e., a vertex adjacent to all previous vertices) is added \cite{Mahadev}.} with clique number $k+1$.
\end{conjecture}

Partial progress on the full Brouwer's conjecture has been made in \cite{CZ,LG}.
In this note, we completely characterize the equality case in \eqref{eq-1-1}.
This, as well as the confirmed Brouwer's conjecture, would yield a complete solution to the full Brouwer's conjecture.

\begin{theorem}\label{thm-1-1}
Let $G$ be a graph with $n$ vertices and $m$ edges. Then for every $k\in\{1,2,\dots,n-1\}$,
\begin{equation}\label{eq-1-2}
  S_k(G)=m+\binom{k+1}{2}
\end{equation}
holds if and only if $G$ is a threshold graph with clique number $k+1$.
\end{theorem}

\begin{remark}\label{rm-1-1}
The sufficiency of Theorem \ref{thm-1-1} has been proven in \cite{LG}.
For the necessity, it is known that if $G$ is a threshold graph and (\ref{eq-1-2}) holds, then $G$ has clique number $k+1$;
see, e.g., the proof of Theorem 3.2 in \cite{CZ}.
So, to complete the proof of Theorem \ref{thm-1-1}, we just need to show that for every graph $G$ with $n$ vertices and $m$ edges,
if (\ref{eq-1-2}) holds, then $G$ is a threshold graph.
This will be done in the next section.
\end{remark}

\section{Proof of Theorem~\ref{thm-1-1}}

We consider the $n$-dimensional Euclidean space $\R^n$ with standard inner product $\langle \cdot, \cdot \rangle$.
For a column vector $\mathbf{x}:=(x_1,x_2,\cdots,x_n)^\mathsf T\in \R^n$, we write $\norm{\mathbf{x}}$ for its Euclidean norm, 
i.e., $\norm{\mathbf{x}}=\sqrt{\langle \mathbf{x}, \mathbf{x} \rangle}$, where the superscript $\mathsf T$ denotes transposition.
The vector $\mathbf{e}_i$ is the $i$th standard basis (column) vector of $\R^n$, and $\one_n$ is the all-ones (column) vector in $\R^n$.
Let $I_n$ be the $n\times n$ identity matrix, and let
\[
  J_n:=\one_n\one_n^{\mathsf T}\,\,\,
  \textrm{and}\,\,\,
  C:=I_n-\frac1nJ_n.
\]
Note that $J_n$ is the $n\times n$ all-ones matrix and $C$ is the orthogonal projection of $\R^n$ onto $\one_n^\perp$,
the orthogonal complement of the subspace of $\R^n$ spanned by~$\one_n$.

Let $P$ be an $n\times n$ orthogonal projection satisfying $P\one_n=\mathbf{0}$.
Then $P^2=P$ and $P^{\mathsf T}=P$.
We define the symmetric matrix $M$ associated with $P$ by
\[
  M_{ii}=0,\,\,\,
  \textrm{and}\,\,\,
  M_{ij}=P_{ii}+P_{jj}-2P_{ij}-1\,\,\, \textrm{for $i\ne j$}.
\]
Since $M_{ij}+1=P_{ii}+P_{jj}-2P_{ij}=\norm{P(\mathbf{e}_i-\mathbf{e}_j)}^2$, we have 
$$-1\leq M_{ij}\leq 1.$$
Let $\mathbf{p}:=(P_{11},P_{22},\dots,P_{nn})^{\mathsf T}$ and $\mathbf{v}=(v_1,v_2,\cdots,v_n)^{\mathsf T}:=CM\one_n$.
It is easy to check that

\begin{equation}\label{eq-2-1}
  \mathbf{v}=nC\mathbf{p}\,\,\,
  \textrm{and}\,\,\,
  v_i-v_j=n(P_{ii}-P_{jj}).
\end{equation}

The following two lemmas, due to Kothari and Tudose~\cite{KT}, are ingredients in the proof of Brouwer's conjecture.

\begin{lemma}[Kothari and Tudose~\cite{KT}, Lemma~5.2]\label{lm-2-1}
For every orthogonal projection $P$ of rank $k$ satisfying $P\one_n=\mathbf{0}$, the following holds
\begin{equation*}
  \frac14\sum_{i\ne j}
  \left[(M_{ij}+1)^2-(P_{ii}-P_{jj})^2\right]=k(k+1).
\end{equation*}
\end{lemma}

\begin{lemma}[Kothari and Tudose~\cite{KT}, Lemma~5.5]\label{lm-2-2}
For every orthogonal projection $P$ of rank $k$ satisfying $P\one_n=\mathbf{0}$, the following holds
\begin{equation*}
 \norm{\mathbf{v}}^2\leq \sum_{i<j}(1-|M_{ij}|)|v_i-v_j|.
\end{equation*}
\end{lemma}

By Lemmas~\ref{lm-2-1} and~\ref{lm-2-2}, we now establish a sharp version of \cite[Lemma~5.6]{KT},
which is a key step towards proving Theorem~\ref{thm-1-1}.
It should be noted that the inequality \eqref{eq-2-3} has been proven by Kothari and Tudose~\cite{KT};
we reprove it here for conveniently discussing the case of equality.
   
For a real number $x$, let $x_+:=\max\{x,0\}$. It is easy to check that
for $-1\leq x\leq 1$,  
\begin{equation}\label{eq-2-2}
  x_+=\frac14\bigl[(x+1)^2-(1-|x|)^2\bigr].
\end{equation}

\begin{lemma}\label{lm-2-3}
For every orthogonal projection $P$ of rank $k$ satisfying $P\one_n=\mathbf{0}$, the following holds
\begin{equation}\label{eq-2-3}
  \sum_{i\ne j}\pos{M_{ij}}\leq k(k+1).
\end{equation}
Moreover, if the equality holds in \eqref{eq-2-3}, then for every pair $(i,j)$ with $i\neq j$,
\begin{equation}\label{eq-2-4}
  1-|M_{ij}|=|P_{ii}-P_{jj}|.
\end{equation}
\end{lemma}

\begin{proof} 
Applying (\ref{eq-2-2}) with $x=M_{ij}$ and summing over all ordered pairs $(i,j)$ with $i\ne j$, we have
\[
  \sum_{i\ne j}\pos{M_{ij}}
  =\frac14\sum_{i\ne j}\bigl[(M_{ij}+1)^2-(1-|M_{ij}|)^2\bigr],
\]
which, together with Lemma~\ref{lm-2-1}, yields that
\begin{equation}\label{eq-2-5}
  k(k+1)-\sum_{i\ne j}\pos{M_{ij}}
  =\frac14\sum_{i\ne j}
  \left[(1-|M_{ij}|)^2-(P_{ii}-P_{jj})^2\right].
\end{equation}
If $\mathbf{v}=\mathbf{0}$, then the right-hand side of \eqref{eq-2-5} is $\frac14\sum_{i\ne j}(1-|M_{ij}|)^2$, and \eqref{eq-2-1} gives $P_{ii}=P_{jj}$ for all $i,j$.
Consequently, \eqref{eq-2-3} holds, and if the equality holds then $1-|M_{ij}|=0=|P_{ii}-P_{jj}|$ for $i\ne j$.

Assume now that $\mathbf{v}\ne \mathbf{0}$.  
By Lemma~\ref{lm-2-2} and the Cauchy--Schwarz inequality, we obtain
\begin{eqnarray}
  \norm{\mathbf{v}}^2&\leq& \sum_{i<j}(1-|M_{ij}|)|v_i-v_j|\label{eq-2-6}\\
  &\leq& \bigg(\sum_{i<j}(1-|M_{ij}|)^2\bigg)^{1/2} \bigg(\sum_{i<j}(v_i-v_j)^2\bigg)^{1/2}.\label{eq-2-6'}
\end{eqnarray}
Furthermore, since $\mathbf{v}\perp\one_n$, applying the well-known Lagrange identity, we get
\[
  \sum_{i<j}(v_i-v_j)^2=n\norm{\mathbf{v}}^2,
\]
which, together with \eqref{eq-2-6'}, yields that
\[
  \sum_{i<j}(1-|M_{ij}|)^2\geq \frac1n\norm{\mathbf{v}}^2,
\]
that is,
\begin{equation}\label{eq-2-7}
  \sum_{i\ne j}(1-|M_{ij}|)^2\geq \frac2n\norm{\mathbf{v}}^2.
\end{equation}
On the other hand, \eqref{eq-2-1} and $\mathbf{v}\perp\one_n$ imply that
\begin{equation}\label{eq-2-8}
  \sum_{i\ne j}(P_{ii}-P_{jj})^2
  =\frac1{n^2}\sum_{i\ne j}(v_i-v_j)^2
  =\frac2n\norm{\mathbf{v}}^2.
\end{equation}
Consequently, \eqref{eq-2-3} follows directly from \eqref{eq-2-5}, \eqref{eq-2-7}, and \eqref{eq-2-8}, as desired.

Moreover, if the equality holds in \eqref{eq-2-3}, then all inequalities in the above argument must be equalities;
in particular, from \eqref{eq-2-6'} we can deduce that there exists a constant $\alpha>0$ such that
\begin{equation}\label{eq-2-9}
1-|M_{ij}|=\alpha |v_i-v_j|\,\,\, \textrm{for $i<j$}.
\end{equation} 
Substituting this into \eqref{eq-2-6} gives
\[
  \norm{\mathbf{v}}^2=\alpha\sum_{i<j}(v_i-v_j)^2=\alpha n\norm{\mathbf{v}}^2,
\]
which implies that $\alpha=1/n$. 
Consequently, \eqref{eq-2-4} follows from \eqref{eq-2-1} and \eqref{eq-2-9}, as required.

This completes the proof of Lemma \ref{lm-2-3}.
\end{proof}

For a given $n\times n$ orthogonal projection $P$ satisfying $P\textbf{1}_n=\textbf{0}$, 
we define the graph $\Gamma(P)$ on $[n]:=\{1,2,\dots,n\}$ for which $\{i,j\}\in E(\Gamma(P))$ if and only if $M_{ij}>0$.

\begin{lemma}\label{lm-2-4}
Let $P$ be an $n\times n$ orthogonal projection satisfying $P\one_n=\mathbf{0}$. 
If \eqref{eq-2-4} holds for every pair $(i,j)$ with $i\ne j$, then $\Gamma(P)$ is a threshold graph.
\end{lemma}

\begin{proof}
We proceed by induction on $n$.  
The case $n=1$ is clear.
Assume that the statement holds for every $(n-1)\times(n-1)$ orthogonal projection satisfying the hypotheses,
and now consider the $n\times n$ orthogonal projection $P$. 
For $1\leq i\leq n$, since $P_{ii}=\mathbf{e}_i^{\mathsf T}P\mathbf{e}_i=\norm{P\mathbf{e}_i}^2$, we have $0\leq P_{ii}\leq 1$.  
Furthermore, if $P_{ii}=1$, then $\norm{P\mathbf{e}_i}^2=1=\norm{\mathbf{e}_i}^2$ and hence, $P\mathbf{e}_i=\mathbf{e}_i$.
However, 
\[
  1 =\langle \mathbf{e}_i,\one_n\rangle
    =\langle P\mathbf{e}_i,\one_n\rangle
    =\langle \mathbf{e}_i,P\one_n\rangle
    =0, 
\]
yielding a contradiction. Thus, we have
\begin{equation} \label{eq-2-10}
0\leq P_{ii}<1\,\,\, \textrm{for $1\leq i\leq n$.}
\end{equation}
This also implies that
\begin{equation} \label{eq-2-11}
M_{ij}\neq 0\,\,\, \textrm{for $i\neq j$.}
\end{equation}
Otherwise, if $M_{ij}=0$, then \eqref{eq-2-4} gives $|P_{ii}-P_{jj}|=1$, contradicting \eqref{eq-2-10}.

We next consider the following two cases.

\textbf{Case 1.}  $P_{ii}>0$ for $1\leq i\leq n$.  

Relabel the vertices of $\Gamma(P)$ so that $P_{nn}$ is maximal.  
Let $N$ denote the set of neighbors of the vertex $n$ and let $\overline{N}:=V(\Gamma(P))\backslash (N\cup \{n\})$.  
By \eqref{eq-2-11} and the definition of $\Gamma(P)$, we obtain
$$
  \textrm{$M_{nj}>0$ for $j\in N$,\, and $M_{nj}<0$ for $j\in \overline{N}$}.
$$
Since $P_{nn}\ge P_{jj}$, \eqref{eq-2-4} gives $1-|M_{nj}|=P_{nn}-P_{jj}$.
Solving this equation with $M_{nj}=P_{nn}+P_{jj}-2P_{nj}-1$, we obtain
\begin{equation*}
\textrm{$P_{nj}=P_{nn}-1$ for $j\in N$,\, and $P_{nj}=P_{jj}$ for $j\in \overline{N}$}.
\end{equation*}
Then, from the last row of $P\one_n=\mathbf{0}$ we see that
\begin{equation*}
\sum_{j\in \overline{N}}P_{jj}=\sum_{j\in \overline{N}}P_{nj}=|N|(1-P_{nn})-P_{nn}.
\end{equation*}
Consequently, the entry $(n,n)$ of $P^2=P$ gives
\begin{equation}
\sum_{j\in \overline{N}}P_{jj}^2=P_{nn}(1-P_{nn})-|N|(1-P_{nn})^2 =-(1-P_{nn})\sum_{j\in \overline{N}}P_{jj}, \label{eq-2-12}
\end{equation}
which implies that $\overline{N}=\varnothing$ and hence, $n$ is a dominating vertex of $\Gamma(P)$.
Otherwise, if $\overline{N}\ne\varnothing$, then $\sum_{j\in \overline{N}}P_{jj}>0$ and $1-P_{nn}>0$, contradicting \eqref{eq-2-12}.

Now, we have $P_{nj}=P_{nn}-1$ for $1\leq j\leq n-1$.  
Considering again the last row of $P\one_n=\mathbf{0}$, we have
$P_{nn}+(n-1)(P_{nn}-1)=0$, which implies that
\[
  \textrm{$P_{nn}=\frac{n-1}{n}$,\, and $P_{nj}=-\frac1n$ for $1\leq j\leq n-1$}.
\]
Thus, $P$ can be written as the following block matrix:
\[
  P=
  \begin{pmatrix}
    B & -\dfrac1n\one_{n-1}\\[2mm]
    -\dfrac1n\one_{n-1}^{\mathsf T} & \dfrac{n-1}{n}
  \end{pmatrix}.
\]
Since $P\one_n=\mathbf{0}$, we have $B\one_{n-1}=\frac1n\one_{n-1}$.
Also, the upper-left block of $P^2=P$ gives 
$$B^2 = B-\frac1{n^2}J_{n-1}.$$
We further define
\[
  \widehat P:=B-\frac1{n(n-1)}J_{n-1}.
\]
Noting that $J_{n-1}^2=(n-1)J_{n-1}$, we can easily check that
\[
  \textrm{$\widehat P\one_{n-1}=\mathbf{0}$\, and\, $\widehat P^2=\widehat P$}.
\]
Also, since both $B$ and $J_{n-1}$ are symmetric, so is $\widehat P$. 
These show that $\widehat P$ is an $(n-1)\times (n-1)$ orthogonal projection satisfying $\widehat P\one_{n-1}=\mathbf{0}$.
Furthermore, observing that for every pair $(i,j)$ in $[n-1]\times [n-1]$,
\[
  \widehat P_{ij}=P_{ij}-\frac{1}{n(n-1)},
\]
we can conclude that
\[
  \widehat P_{ii}-\widehat P_{jj}=P_{ii}-P_{jj},
\]
and
\[
  \textrm{$\widehat M_{i,j}:=\widehat P_{ii}+\widehat P_{jj}-2\widehat P_{ij}-1=P_{ii}+P_{jj}-2P_{ij}-1=M_{ij}$\quad $(i\neq j)$}.
\]
These mean that \eqref{eq-2-4} also holds for $\widehat P$, 
and the graph $\Gamma(\widehat P)$ is exactly the subgraph of $\Gamma(P)$ induced by the vertices in $[n-1]$.  
Consequently, the induction hypothesis gives that $\Gamma(\widehat P)$ is a threshold graph and hence, so is the graph $\Gamma(P)$,
because it can be obtained from the threshold graph $\Gamma(\widehat P)$ by adding the dominating vertex $n$.

\textbf{Case 2.} $P_{ii}=0$ for some $i\in [n]$.

Relabel the vertices of $\Gamma(P)$ so that $P_{nn}=0$. 
Then, we have $\norm{P\mathbf{e}_n}^2=P_{nn}=0$, which implies that the $n$th row and column of $P$ are $\mathbf{0}$. 
Hence, for every $j\in [n-1]$, by \eqref{eq-2-10} we get
\[
  M_{nj}=P_{nn}+P_{jj}-2P_{nj}-1=P_{jj}-1<0,
\]
which means that the vertex $n$ of $\Gamma(P)$ is isolated.

Let $P'$ be the matrix obtained from $P$ by deleting its $n$th row and column.  
It is easy to check that $P'$ is an $(n-1)\times(n-1)$ orthogonal projection satisfying $P'\one_{n-1}=\mathbf{0}$.
Furthermore, we observe that $P'_{ij}=P_{ij}$ holds for every pair $(i,j)$ in $[n-1]\times [n-1]$ and hence,
$$M'_{ij}:=P'_{ii}+P'_{jj}-2P'_{ij}-1=M_{ij}\,\,\, (i\ne j).$$
Now, by the same argument as in Case 1, we can conclude that the graph $\Gamma(P)$ is a threshold graph,
because it can be constructed from the threshold graph $\Gamma(P')$ by adding the isolated vertex $n$.

This completes the proof of Lemma \ref{lm-2-4}.
\end{proof}

We are now ready to give a proof of Theorem \ref{thm-1-1}.

\begin{proof}[Proof of Theorem~\ref{thm-1-1}]
As mentioned in Remark \ref{rm-1-1}, we just need to show that 
the equality \eqref{eq-1-2} forces the graph $G$ to be threshold.

Restrict $L(G)$ to $\one_n^\perp$, and let $P$ be the orthogonal projection onto a subspace of $\R^n$ spanned by the top $k$ eigenvectors of $L(G)$,
that is, $k$ linearly independent eigenvectors of $L(G)$ corresponding to $\mu_1(G),\mu_2(G),\cdots,\mu_k(G)$.
Then $P$ has rank $k$ and satisfies $P\textbf{1}_n=\textbf{0}$. 
As shown in the proof of \cite[Theorem 3.1]{KT}, we have 
\begin{eqnarray}
  S_k(G)-m
  =\sum_{\{i,j\}\in E(G)}M_{ij}
  \leq \sum_{\{i,j\}\in E(G)}\pos{M_{ij}}
  \leq \sum_{1\le i<j\le n}\pos{M_{ij}}
  \leq \binom{k+1}{2}.\label{eq-2-13}
\end{eqnarray}

Assume now that the equality \eqref{eq-1-2} holds. 
Then all the inequalities in \eqref{eq-2-13} must be equalities;
in particular, equality in the first inequality yields that $M_{ij}>0$ for $\{i,j\}\in E(G)$,
while the second forces $(M_{ij})_{+}=0$ for $\{i,j\}\notin E(G)$ and hence, $M_{ij}<0$ for $\{i,j\}\notin E(G)$
(since \eqref{eq-2-11} gives $M_{ij}\neq 0$ for $i\neq j$).
These imply that $G$ is isomorphic to $\Gamma(P)$ and consequently, 
from Lemma~\ref{lm-2-4} it follows that $G$ is threshold, completing the proof. 
\end{proof}

\textbf{Note.}
During the preparation of our manuscript, Professor Xiao-Dong Zhang
kindly shared with us their work~\cite{CCYZ} on the full Brouwer's Laplacian conjecture. 
Comparing our work with~theirs, we find that Lemma~\ref{lm-2-3} overlaps with Lemma~15 in~\cite{CCYZ}, 
both of which rely on the projection method of Kothari and Tudose \cite{KT}.
The main difference is that we show directly that the equality case can occur only for threshold graphs by Lemma \ref{lm-2-3},
whereas the work \cite{CCYZ} for split graphs. \\

\textbf{Declaration of AI Use.}
We used GPT 5.5 Pro to simplify our proof for Lemma \ref{lm-2-4}.

\end{document}